\documentclass[20pt]{amsart}
\usepackage{amssymb}
\usepackage{amsfonts}
\usepackage{amsmath}
\usepackage{mathrsfs}
\usepackage{amsthm}
\newtheorem{theorem}{Theorem}[section]
\newtheorem{lemma}[theorem]{Lemma}

\theoremstyle{definition}

\theoremstyle{remark}

\newcommand{\va}{\mathbf a}
\newcommand{\vb}{\mathbf b}
\newcommand{\vz}{\mathbf z}

\newcommand{\Z}{\mathbb{Z}}

\newcommand{\abs}[1]{\lvert#1\rvert}

\title{On generalizations of  $p$-sets and their  applications}
\thanks{Heng Zhou was supported by the National
Natural Science Foundation of China (No. 61602341);
 Zhiqiang Xu was supported  by NSFC grant (11422113,  91630203, 11331012) and by National Basic Research Program of China (973 Program 2015CB856000).
}
\author{Heng Zhou}
\address{School of Sciences, Tianjin Polytechnic University, Tianjin 300160, China \newline{\tt E-mail address: zhouheng7598@sina.com.cn}}
\author{Zhiqiang Xu}
\address{LSEC, Institute of Computational Mathematics, Academy of Mathematics and System Sciences,
Chinese Academy of Sciences, Beijing 100190, China \newline{\tt E-mail address: xuzq@lsec.cc.ac.cn}}
\date{}
\begin{document}
\maketitle
\begin{center}
\parbox{11cm}{{\it Abstract}:\quad
The $p$-set, which is in a simple analytic form, is well distributed in unit cubes. The well-known Weil's  exponential sum theorem presents an upper bound of the exponential sum over the $p$-set. Based on the result, one shows that the $p$-set performs well in numerical integration, in compressed sensing as well as in UQ. However, $p$-set is somewhat rigid since the cardinality of the $p$-set is a prime $p$ and the set  only depends on the prime number $p$. The purpose  of this paper is to present  generalizations  of $p$-sets, say $\mathcal{P}_{d,p}^{\va,\epsilon}$, which is more flexible. Particularly, when a prime number $p$ is given, we have many different choices of the new $p$-sets. Under the assumption that Goldbach conjecture holds, for any even number $m$, we present  a point set, say ${\mathcal L}_{p,q}$, with cardinality $m-1$   by combining two different new $p$-sets, which overcomes a major bottleneck of the $p$-set. We also present the upper bounds of the exponential sums over  $\mathcal{P}_{d,p}^{\va,\epsilon}$ and
${\mathcal L}_{p,q}$, which imply these sets have many potential applications.
}
\end{center}
\vspace{0.3cm}
\begin{center}
\parbox{11cm}{{\it Key words and phrases}\quad $p$-set; Deterministic sampling; Numerical integral; Exponential  sum}
\end{center}
\vspace{0.3cm}
\begin{center}
\parbox{11cm}{{\it AMS Subject Classification 2000}\quad  11K38, 65C05, 11L05, 41A10, 65D32}
\end{center}
\mbox{}\\[10pt]
\section{Introduction}

\subsection{$p$-set}
Let $p$ be a prime number. We consider the point set
\[
\mathcal{P}_{d,p}=\left\{\mathbf{x}_{0}, \ldots, \mathbf{x}_{p-1}\right\}\subset[0, 1)^{d}
\]
where
\[
\mathbf{x}_{j}=\Big(\left\{\frac{j}{p}\right\}, \left\{\frac{j^{2}}{p}\right\}, \ldots, \left\{\frac{j^{d}}{p}\right\}\Big)\in [0,1)^d,\quad j\in \mathbb{Z}_{p},
\]
$\mathbb{Z}_{p}:=\left\{0, 1, \ldots, p-1\right\}$ and $\left\{x\right\}$ is the fractional part of $x$ for a
nonnegative real number $x$. The point set $\mathcal{P}_{d,p}$ is called $p$-set and was introduced by Korobov \cite{pset} and  Hua-Wang \cite{HW}. Recently, $p$-set attracts much attention since its advantage in numerical integration \cite{Dick3}, in the recovery of sparse trigonometric polynomials \cite{Xu} and in the UQ \cite{XZ}. In \cite{Dick3}, Dick presents a numerical integration formula based on $\mathcal{P}_{d,p}$ with showing the error bound of the formula depends only polynomially on the dimension $d$. In \cite{Xu}, Xu uses $\mathcal{P}_{d,p}$ to construct the deterministic sampling points of sparse trigonometric polynomials and show the sampling matrix corresponding to $\mathcal{P}_{d,p}$ has the almost optimal coherence. And hence, $\mathcal{P}_{d,p}$ has a good performance for the recovery of sparse trigonometric polynomials.

\subsection{ Extensions of $p$-set: $\mathcal{P}^{\va, \epsilon}_{d, p}$ and $\mathcal{L}_{p, q}$}
The $p$-set is in a simple analytic form and hence it is easy to be generated by computer. However, the $p$-set is somewhat rigid with the point
set only depending on a prime number $p$. If the function values at some points in $p$-set are not easy to be obtained, one has to change the prime
number $p$ to obtain a new point set which has the different cardinality with the previous one. Hence, in practical application, it will be better that one has many different choices.  We next introduce a generalization  of $p$-set.

Let
\[
\mathbb{Z}_{p}^{d}:=\left\{\va=(a_{1}, \ldots, a_{d})\in\mathbb{Z}^{d}: a_{j}\in\mathbb{Z}_{p}, j=1, \ldots, d\right\}.
\]
 Suppose that $\va=(a_{1}, \ldots, a_{d})\in \mathbb{Z}^{d}_{p}$ and $\epsilon=(\epsilon_{1},
\ldots, \epsilon_{d-1})\in \left\{0, 1\right\}^{d-1}$. We set
\begin{equation}\label{eq:pset}
\mathcal{P}^{\va, \epsilon}_{d, p}:=\left\{\mathbf{x}_{j}^{\va,\epsilon}:j\in\mathbb{Z}_{p}\right\}
\end{equation}
where
$$
\mathbf{x}_{j}^{\va,\epsilon}:=\Big(\left\{\frac{a_{1}j}{p}\right\}, \left\{\frac{a_{1}'j+a_{2}j^{2}}{p}\right\}, \ldots, \left\{
\frac{\sum_{h=1}^{d-1}a_{h}'j^{h}+a_{d}j^{d}}{p}\right\}\Big)\in [0,1)^d
$$
and
 $a_{k}'=\epsilon_{k}a_{k}, k=1, \ldots, d-1$.
 We call $\mathcal{P}^{\va, \epsilon}_{d, p}$ as the $p$-set associating with
the parameter $\va$ and $\epsilon$. If we take $\va=(1, \ldots, 1)$ and $\epsilon=(0, \ldots, 0)$, then $\mathcal{P}^{\va, \epsilon}_{d, p}$ is reduced to the classical $p$-set.

 The $p$-set $\mathcal{P}^{\va, \epsilon}_{d, p}$ associating with the parameters $\va, \epsilon$ is more flexible. Given the prime number $p$, one can generate various point sets by changing the parameters $\va$ and $\epsilon$ with presenting an option set when the cardinality  $p$ is given.

Note that the cardinality of both $\mathcal{P}^{\va, \epsilon}_{d, p}$ and $\mathcal{P}_{d, p}$ is prime.
Since the  distance between adjacent prime can be very large, the cardinality of  $p$-set does not change smoothly.
Using the set $\mathcal{P}^{\va, \epsilon}_{d, p}$, we next present a set with the cardinality being   odd number. Suppose that $m\in 2\mathbb{Z}$ is given. The Goldbach conjecture, which is
one of the best-known unsolved problem in number theory, says that $m$ can be written as the sum of two primes, i.e., $m=p+q$ where $p$ and $q$
are prime numbers. One has verified the conjecture up to $m\leq 4\cdot 10^{14}$ which is enough for practical application. We next suppose that $m=p+q$
with $p$ and $q$ being prime numbers. We set
\begin{equation}\label{eq:sp1}
\mathcal{L}_{p, q}:=\left\{
\begin{array}{c}
\mathcal{P}_{d,p}\cup\mathcal{P}_{d,q}, \quad p\neq q\\
\mathcal{P}^{\va, \epsilon'}_{d, p}\cup
\mathcal{P}^{\vb, \epsilon''}_{d, p}, \quad p=q,
\end{array}\right.
\end{equation}
where $\mathcal{P}^{\va, \epsilon'}_{d, p}$ and $\mathcal{P}^{\vb, \epsilon''}_{d, p}$ are the $p$-sets that we have defined above and $\va,\vb\in \mathbb{Z}_{p}^{d}, \epsilon',\epsilon''\in \left\{0,1\right\}^{d-1}$. We call $\mathcal{L}_{p, q}$ the  $(p,q)$-set.  As shown later, $\mathcal{P}_{d,p}\cap \mathcal{P}_{d,q}=\left\{(0,\ldots,0)\right\}$ provided $p\neq q$. We can choose $\va,\vb, \epsilon'$ and $\epsilon''$  so that $\mathcal{P}^{\va, \epsilon'}_{d, p}\cap
\mathcal{P}^{\vb, \epsilon''}_{d, p}=\left\{(0,\ldots,0)\right\}$. Hence, under the assumption of Goldbach conjecture, for any odd number, says $m-1$, there exist $p,q$ so that $|\mathcal{L}_{p, q}|=p+q-1=m-1$.

We would like to mention the following point sets with cardinality $p^2$ \cite{pset,HW} :
\begin{equation}\label{eq:p2set}
\begin{aligned}
{\mathcal Q}_{p^2,d}&=\left\{\vz_j : j=0,\ldots,p^2-1\right\},\, \mathbf{z}_{j}=\Big(\left\{\frac{j}{p^2}\right\}, \left\{\frac{j^{2}}{p^2}\right\}, \ldots, \left\{\frac{j^{d}}{p^2}\right\}\Big)\in [0,1)^d;\\
{\mathcal R}_{p^2,d}&=\left\{\vz_{j,k} : j,k=0,\ldots,p-1\right\},\, \vz_{j,k}=\Big(\left\{\frac{k}{p} \right\},\left\{\frac{jk}{p} \right\},\ldots,\left\{\frac{j^{d-1}k}{p} \right\}\Big)\in [0,1)^d.
\end{aligned}
\end{equation}
The weighted star discrepancy of ${\mathcal Q}_{p^2,d}$ and ${\mathcal R}_{p^2,d}$ is given in \cite{Dick2}. Using a similar method with above, we can  generalize   ${\mathcal Q}_{p^2,d}$ and ${\mathcal R}_{p^2,d}$ to
 ${\mathcal Q}_{p^2,d}^{\va,\epsilon}$ and ${\mathcal R}_{p^2,d}^{\va,\epsilon}$, respectively. We will introduce it in Section 2.3 in detail.

\subsection{Organization}
In Section 2, we present the upper bounds of  the exponential sums over  $\mathcal{P}^{\va, \epsilon}_{d, p}$ and $\mathcal{L}_{p, q}$. Particularly, we present the condition under which
$\abs{\mathcal{L}_{p, q}}=p+q-1$ and also  prove  that
$\mathcal{P}^{\va, \epsilon}_{d, p}\cap \mathcal{P}^{\vb, \epsilon}_{d, p}=\left\{(0,\ldots,0)\right\}$ when $\mathcal{P}^{\va, \epsilon}_{d, p}\neq \mathcal{P}^{\vb, \epsilon}_{d, p}$. We furthermore consider the generalization  of the point sets ${\mathcal Q}_{p^2,d}$ and ${\mathcal R}_{p^2,d}$ and present the upper bounds of exponential sums over the new sets.  The results in Section 2 show that
  the  point sets presented in this paper  have many potential applications in various  areas.
In Section 3, we choose  $\mathcal{L}_{p, q}$ as a deterministic sampling set  for the recovery of sparse trigonometric polynomials and then show their performance.

\section{The exponential sums over $\mathcal{P}^{\va, \epsilon}_{d, p}$ and $\mathcal{L}_{p, q}$}

The aim of this section is to present the exponential sums over $\mathcal{P}^{\va, \epsilon}_{d, p}$ and $\mathcal{L}_{p, q}$. To this end, we first introduce the well-known Weil's formula, which plays a key role in our proof.
\begin{theorem}\label{theorem:re1}
\cite{weil} Suppose that $p$ is a prime number. Suppose $f(x)=\sum_{h=1}^{d}m_{h}x^{h}$ with $m_{h}\in\mathbb{Z}$ $(h=1, \ldots, d)$ and there is a $j\in \left\{1,2, \ldots d\right\}$,
satisfying $p \nmid m_{j}$. Then
\begin{equation}
\Big|\sum_{x=1}^{p}e^{\frac{2\pi \mathbf{i} f(x)}{p}}\Big|\leq (d-1)\sqrt{p}.\nonumber
\end{equation}
\end{theorem}

\subsection{The exponential sum over $\mathcal{P}^{\va, \epsilon}_{d, p}$ }
Recall that
\[
\mathcal{P}^{\va, \epsilon}_{d, p}\,\,:=\,\,\left\{\mathbf{x}_{j}^{\va,\epsilon}:j\in\mathbb{Z}_{p}\right\}
\]
and
\[
\mathbf{x}_{j}^{\va,\epsilon}=\Big(\left\{\frac{a_{1}j}{p}\right\}, \left\{\frac{a_{1}'j+a_{2}j^{2}}{p}\right\}, \ldots, \left\{
\frac{\sum_{h=1}^{d-1}a_{h}'j^{h}+a_{d}j^{d}}{p}\right\}\Big)\in [0,1)^d
\]
where $\va=(a_1,\ldots,a_d)\in[1, p-1]^{d}\cap\mathbb{Z}^{d}$, $a_j'=\epsilon_ja_j$ and
$\epsilon=(\epsilon_{1},\ldots, \epsilon_{d-1})\in \left\{0, 1\right\}^{d-1}$.
Note that $|\mathcal{P}^{\va, \epsilon}_{d, p}|=p$.  We next show the exponential sum formula over
$\mathcal{P}^{\va, \epsilon}_{d, p}$.
\begin{theorem}\label{th:weil1}
For any $\mathbf{k}\in[-p+1, p-1]^{d}\cap\mathbb{Z}^{d}$ and $\mathbf{k}\neq 0$, we have
\begin{equation}\label{eq:sum1}
\left|\sum_{\mathbf{x}\in \mathcal{P}^{\va, \epsilon}_{d, p}}\exp(2\pi\mathbf{i}\mathbf{k}\cdot \mathbf{x})\right|
=\Big|\sum_{j=0}^{p-1}\exp(2\pi\mathbf{i}\mathbf{k}\cdot\mathbf{x}_{j}^{\va,\epsilon})\Big|\leq (d-1)\sqrt{p}.
\end{equation}
\end{theorem}
\begin{proof}
Set
\begin{equation*}
g(j)=\sum_{\ell=1}^{d}c_{\ell}j^{\ell}
\end{equation*}
where $c_{\ell}=k_{\ell}a_{\ell}+k_{\ell+1}a_{\ell}'+\cdots+k_{d}a_{\ell}'$.
We set  $j_0:=\max\left\{\ell:k_{\ell}\neq 0\right\}$. Then $c_{j_0}=k_{j_0}a_{j_0}$ and we have $p\nmid c_{j_0}$.
According to  Theorem \ref{theorem:re1}, we obtain that
\begin{equation*}
\Big|\sum_{j=0}^{p-1}\exp(2\pi\mathbf{i}\mathbf{k}\cdot\mathbf{x}_{j}^{\va,\epsilon})\Big|
 =\Big|\sum_{j=0}^{p-1}\exp\Big(2\pi\mathbf{i}\frac{g(j)}{p}\Big)\Big|\leq (d-1)\sqrt{p}.
\end{equation*}

\end{proof}

\subsection{The exponential sum over  $\mathcal{L}_{p, q}$}
To this end, we  consider the cardinality of $\mathcal{L}_{p,q}$. A simple observation is
that $\abs{\mathcal{L}_{p,q}}\leq p+q-1$. We would like to present the condition under which
$\abs{\mathcal{L}_{p,q}}= p+q-1$.
 We first consider the case where $p\neq q$.
\begin{theorem}
Suppose that $p$ and $q$ are two distinct prime numbers. Then $|\mathcal{L}_{p,q}|=p+q-1$.
\end{theorem}
\begin{proof}
According to (\ref{eq:sp1}), to this end, we just need show that
\[
\mathcal{P}_{d, p}\cap\mathcal{P}_{d, q}=\left\{(0, \ldots, 0)\right\}.
\]
 We prove it by contradiction. Assume that $\mathcal{P}_{d, p}\cap\mathcal{P}_{d, q}\neq\left\{(0, \ldots, 0)\right\}$, and then there exists
$j\in\mathbb{Z}_{p}^{*}:=\mathbb{Z}_{p}\setminus\left\{0\right\}$ and $k\in\mathbb{Z}_{q}^{*}$  so that $\left\{\frac{j^{i}}{p}\right\}=\left\{\frac{k^{i}}{q}\right\}$, $i=1, \ldots, d$.
Particularly, we have $\frac{j}{p}=\frac{k}{q}$, which is equivalent to $jq=kp$. Since $p$ and $q$ are different prime numbers, $j\in\mathbb{Z}_{p}^{*}$ and $k\in\mathbb{Z}_{q}^{*}$, we have $j|k$ and $k|j$, which means $j=k$ and hence  $p=q$.  A contradiction.
\end{proof}

We next consider the case where $p=q$, i.e., ${\mathcal L}_{p,q}=\mathcal{P}^{\va, \epsilon'}_{d, p}\cup
\mathcal{P}^{\vb, \epsilon''}_{d, p}$. For the case where $\epsilon'=\epsilon''$,
we have
\begin{theorem}\label{th:deng}
Suppose that $\epsilon\in\left\{0, 1\right\}^{d-1}$  is a fixed vector and $\va=(a_1,\ldots,a_d)\in\mathbb{Z}_{p}^{d}$, $\vb=(b_1,\ldots,b_d)\in\mathbb{Z}_{p}^{d}$.
\begin{enumerate}
\item $\mathcal{P}^{\va, \epsilon}_{d, p}=\mathcal{P}^{\vb, \epsilon}_{d, p}$ if and only if there exists $c\in\mathbb{Z}_{p}^{*}$
such that
\[
b_{j}c^{j}\equiv a_{j}\quad (\mathrm{mod}\quad p) \quad for\quad  j=1,\ldots,d.
\]
\item  If
$\mathcal{P}^{\va, \epsilon}_{d, p}\neq\mathcal{P}^{\vb, \epsilon}_{d, p}$, then $\mathcal{P}^{\va, \epsilon}_{d, p}\cap\mathcal{P}^{\vb, \epsilon}_{d, p}=\left\{(0, \ldots, 0)\right\}$.
\end{enumerate}
\end{theorem}
\begin{proof}
(1) We first suppose that there exists $c\in\mathbb{Z}_{p}^{*}$ such that $b_{j}c^{j}\equiv a_{j}\,\,(\mathrm{mod} \quad p)$ for $j\in \left\{1,\ldots,d\right\}$. Recall that
$$
\mathcal{P}^{\va, \epsilon}_{d, p}=\left\{\mathbf{x}_{j}^{\va,\epsilon} : j\in\mathbb{Z}_{p}\right\},
$$
where
\[
\mathbf{x}_{j}^{\va,\epsilon}=\Big(\left\{\frac{a_{1}j}{p}\right\}, \left\{\frac{a_{1}'j+a_{2}j^{2}}{p}\right\}, \ldots, \left\{
\frac{\sum_{h=1}^{d-1}a_{h}'j^{h}+a_{d}j^{d}}{p}\right\}\Big).
\]
For any $j_{0}\in \mathbb{Z}_{p}$, we take $k_{0}\equiv cj_{0}\quad(\mathrm{mod}\quad p)$. Then
\begin{eqnarray}
\mathbf{x}_{j_{0}}^{\va,\epsilon} &=& \Big(\left\{\frac{a_{1}j_{0}}{p}\right\}, \left\{\frac{a_{1}'j_{0}+a_{2}j_{0}^{2}}{p}\right\}, \ldots, \left\{
\frac{\sum_{h=1}^{d-1}a_{h}'j_{0}^{h}+a_{d}j_{0}^{d}}{p}\right\}\Big)\nonumber\\
&=& \Big(\left\{\frac{b_{1}cj_{0}}{p}\right\}, \left\{\frac{b_{1}'cj_{0}+b_{2}c^{2}j_{0}^{2}}{p}\right\}, \ldots, \left\{
\frac{\sum_{h=1}^{d-1}b_{h}'c^{h}j_{0}^{h}+b_{d}c^{d}j_{0}^{d}}{p}\right\}\Big)\nonumber\\
&=& \Big(\left\{\frac{b_{1}k_{0}}{p}\right\}, \left\{\frac{b_{1}'k_{0}+b_{2}k_{0}^{2}}{p}\right\}, \ldots, \left\{
\frac{\sum_{h=1}^{d-1}b_{h}'k_{0}^{h}+b_{d}k_{0}^{d}}{p}\right\}\Big)\nonumber\\
&=& \mathbf{x}_{k_{0}}^{\vb,\epsilon},\nonumber
\end{eqnarray}
which implies that
$$\mathcal{P}^{\va, \epsilon}_{d, p}\,\,\subseteq\,\, \mathcal{P}^{\vb, \epsilon}_{d, p}.$$
Here we use $b_{j}'c^{j}\equiv a_{j}'\,\, (\mathrm{mod}\,\, p)$ which follows from $b_{j}c^{j}\equiv a_{j}\,\,(\mathrm{mod}\,\, p)$.

Since $p$ is a prime number, there exists $c^{-1}\in\mathbb{Z}_{p}^{*}$ so that $c^{-1}c\equiv 1\quad(\mathrm{mod}\,\, p)$. Then we have
 $b_{j}\equiv  a_{j} c^{-j}\quad(\mathrm{mod}\quad p), j=1, \ldots, d$. Then, similarly, for
any $j_{0}\in\mathbb{Z}_{p}$,
$$
\mathbf{x}_{j_{0}}^{\vb,\epsilon}\,\,=\,\,\mathbf{x}_{c^{-1}j_{0}}^{\va,\epsilon},
$$
which implies that
$$
\mathcal{P}^{\vb, \epsilon}_{d, p}\,\,\subseteq\,\, \mathcal{P}^{\va, \epsilon}_{d, p}.
$$
Then we arrive at
$$
\mathcal{P}^{\va, \epsilon}_{d, p}=\mathcal{P}^{\vb, \epsilon}_{d, p}.
$$
We next suppose that $\mathcal{P}^{\va, \epsilon}_{d, p}=\mathcal{P}^{\vb, \epsilon}_{d, p}$. Then there exist
$j_{0}, k_{0}\in\mathbb{Z}_{p}^{*}$ so that $\mathbf{x}_{j_{0}}^{\va}=\mathbf{x}_{k_{0}}^{\vb}$, i.e.,
\begin{equation}\label{eq:theorem11}
a_{1}j_{0}\equiv b_{1}k_{0}\quad (\mathrm{mod} \quad p),
\end{equation}
\begin{equation}\label{eq:theorem12}
a_{1}'j_{0}+a_{2}j_{0}^{2}\equiv b_{1}'k_{0}+b_{2}k_{0}^{2}\quad (\mathrm{mod} \quad p),
\end{equation}
\begin{eqnarray}
&&\vdots\nonumber
\end{eqnarray}
\begin{equation}
\sum_{h=1}^{d-1}a_{h}'j_{0}^{h}+a_{d}j_{0}^{d}\equiv \sum_{h=1}^{d-1}b_{h}'k_{0}^{h}+b_{d}k_{0}^{d}\quad (\mathrm{mod} \quad p).\nonumber
\end{equation}
We set $c:\equiv k_{0}j_{0}^{-1}\quad (\mathrm{mod} \quad p)$, where $j_{0}^{-1}\in\mathbb{Z}_{p}$ so that $j_{0}^{-1}j_{0}\equiv 1\quad (\mathrm{mod} \quad p)$. Then (\ref{eq:theorem11}) implies that
$a_{1}\equiv b_{1}c\quad (\mathrm{mod} \quad p)$. Combining (\ref{eq:theorem11}) and (\ref{eq:theorem12}), we have $a_{2}j_{0}^{2}\equiv b_{2}k_{0}^{2}\quad (\mathrm{mod} \quad p)$ which implies that $a_{2}\equiv b_{2}c^{2}\quad (\mathrm{mod} \quad p)$. Similarly, we can
obtain that $a_{j}\equiv b_{j}c^{j}\quad (\mathrm{mod} \quad p)$ for $3\leq j\leq d$.

(2) We prove it by contradiction. Assume that $\mathcal{P}_{d, p}\cap\mathcal{P}_{d, q}\neq\left\{(0, \ldots, 0)\right\}$, and then there exist
$j_{0}, k_{0}\in\mathbb{Z}_{p}^{*}$ so that $\mathbf{x}_{j_{0}}^{\va,\epsilon}=\mathbf{x}_{k_{0}}^{\vb,\epsilon}$. Similarly with the above proof, we can
find a $c:\equiv k_{0}j_{0}^{-1}\quad (\mathrm{mod} \quad p)$ so that $a_{j}\equiv b_{j}c^{j}\quad (\mathrm{mod} \quad p)$ for $j=1, \ldots, d$.
It leads to $\mathcal{P}^{\va, \epsilon}_{d, p}=\mathcal{P}^{\vb, \epsilon}_{d, p}$ by (1) of Theorem 2.4, which is impossible by the assumption in (2).
\end{proof}

We next consider the case where  $\epsilon' \neq \epsilon''$.
\begin{theorem}\label{th:cap}
Suppose that $\epsilon', \epsilon''\in\left\{0, 1\right\}^{d-1}$ with $\epsilon'\neq\epsilon''$.
Set
\[
{\mathcal Z}\,\,:=\,\, \left\{j: \epsilon_{j}'\neq \epsilon_{j}'' \,\,\text{and}\,\, a_{j}^{2}+b_{j}^{2}\neq 0, 1\leq j\leq d-1\right\},
\]
and
\[
\ell_0:=\min \left\{j: j\in \mathcal{Z}\right\}.
\]
Then the followings hold.
\begin{enumerate}
\item $\mathcal{P}^{\va, \epsilon'}_{d, p}=\mathcal{P}^{\vb, \epsilon''}_{d, p}$ if and only if there exists a
 $c\in\mathbb{Z}_{p}^{*}$ so that
\begin{equation}\label{eq:cond1}
\begin{aligned}
& a_{j}\equiv b_{j}c^{j}\quad (\mathrm{mod}\quad p) \quad for \quad \epsilon_{j}'=\epsilon_{j}''\quad or\quad j=d, \\
& a_{j}=b_{j}=0\quad for\quad \epsilon_{j}'\neq\epsilon_{j}'',
\end{aligned}
\end{equation}

 where $\va, \vb\in\mathbb{Z}_{p}^{d}\setminus\left\{(0, \ldots, 0)\right\}$ with $a_{1}\neq 0$ being given.
\item Assume that $\mathcal{P}^{\va, \epsilon'}_{d, p}\neq\mathcal{P}^{\vb, \epsilon''}_{d, p}$ where $\va, \vb\in\mathbb{Z}_{p}^{d}$.
If ${\mathcal Z}=\emptyset $ then we  have
\[
\mathcal{P}^{\va, \epsilon'}_{d, p}\cap\mathcal{P}^{\vb, \epsilon''}_{d, p}=\left\{(0,\ldots,0)\right\}.
\]
If ${\mathcal Z} \neq \emptyset$, then
$|{\mathcal{P}^{\va, \epsilon'}_{d, p}\cap\mathcal{P}^{\vb, \epsilon''}_{d, p}}|\leq r+1$, where
\[
r=\min\left\{j: a_{j}^{2}+b_{j}^{2}\neq 0\right\}.
\]

\item
 Assume that ${\mathcal Z}\neq \emptyset$, and $a_1\neq 0$.
If $a_{\ell_0+1}b_{1}^{\ell_0+1}\equiv a_{1}^{\ell_0+1}b_{\ell_0+1}\quad(\mathrm{mod} \quad p)$, then $\mathcal{P}^{\va, \epsilon'}_{d, p}\cap\mathcal{P}^{\vb, \epsilon''}_{d, p}=\left\{(0, \ldots, 0)\right\}$.
\end{enumerate}
\end{theorem}
 \begin{proof}
 (1)
 We assume that (\ref{eq:cond1}) holds.
 Take
 \begin{equation}
\epsilon_j=
\left\{
\begin{array}{c}
\epsilon_j',\quad \text{ if }\epsilon_j'=\epsilon_j''\\
0, \quad \text{ if }\epsilon_j'\neq \epsilon_j''
\end{array}\right. .
\nonumber
\end{equation}
 Noting that $a_j=b_j=0$ provided $\epsilon_j'\neq \epsilon_j''$, we have
 $\mathcal{P}^{\va, \epsilon}_{d, p}=\mathcal{P}^{\va, \epsilon'}_{d, p}$ and $\mathcal{P}^{\vb, \epsilon}_{d, p}=\mathcal{P}^{\vb, \epsilon''}_{d, p}$. Theorem \ref{th:deng} implies that
 $\mathcal{P}^{\va, \epsilon}_{d, p}=\mathcal{P}^{\vb, \epsilon}_{d, p}$ and hence
 $\mathcal{P}^{\va, \epsilon'}_{d, p}=\mathcal{P}^{\vb, \epsilon''}_{d, p}$.

 We next assume that $\mathcal{P}^{\va, \epsilon'}_{d, p}=\mathcal{P}^{\vb, \epsilon''}_{d, p}$
which is equivalent to that there exists
a permutation $\left\{k_{0}, k_{1}, \ldots, k_{p-1}\right\}$ of $\left\{0, 1, \ldots, p-1\right\}$ so that
\begin{equation}\label{eq:6}
\mathbf{x}_{j}^{\va,\epsilon'}=\mathbf{x}_{k_{j}}^{\vb,\epsilon''},\quad j=0, 1, \ldots, p-1.
\end{equation}
This is equivalent to
\begin{equation}\label{eq:theorem31}
\begin{aligned}
    a_{1}&\equiv b_{1}k_{1}\quad (\mathrm{mod}\quad p)\\
    2a_{1}&\equiv b_{1}k_{2}\quad (\mathrm{mod}\quad p)\\
    &\vdots\\
  (p-1)a_{1}&\equiv b_{1}k_{p-1}\quad (\mathrm{mod}\quad p)
  \end{aligned}
\end{equation}
and
\begin{equation}\label{eq:theorem32}
\begin{aligned}
\sum_{h=1}^{i-1} a_{h}'+a_{i} &\equiv \sum_{h=1}^{i-1}b_{h}'k_{1}^{h}+b_{i}k_{1}^{i}\quad (\mathrm{mod} \quad p)\\
\sum_{h=1}^{i-1}2^{h}a_{h}'+a_{i}2^{i}&\equiv \sum_{h=1}^{i-1}b_{h}'k_{2}^{h}+b_{i}k_{2}^{i}\quad (\mathrm{mod}\quad p)\\
\vdots\\ \sum_{h=1}^{i-1}(p-1)^{h}a_{h}'+a_{i}(p-1)^{i}&\equiv \sum_{h=1}^{i-1}b_{h}'k_{p-1}^{h}+b_{i}k_{p-1}^{i}\quad (\mathrm{mod}\quad p),
\end{aligned}
\end{equation}
for $i=2, \ldots, d$. Since $a_{1}\neq 0$, by (\ref{eq:theorem31}) we have
\begin{equation}\label{eq:theorem33}
\begin{aligned}
a_{1}&\equiv b_{1}k_{1}\quad (\mathrm{mod}\,\, p)\\
k_{2}&\equiv 2k_{1}\quad (\mathrm{mod}\,\, p)\\
&\vdots\\
k_{p-1}&\equiv (p-1)k_{1}\quad (\mathrm{mod}\,\, p).
\end{aligned}
\end{equation}
Set $j_0:=\min\left\{i: \epsilon_{i}'\neq \epsilon_{i}''\right\}$. Using the same argument with the one in Theorem \ref{th:deng} we have $a_{i}\equiv b_{i}k_{1}^{i}\quad(\mathrm{mod} \quad p)$, $i=1, \ldots, j_0$. Combining (\ref{eq:theorem32}) for $i=j_0+1$ and (\ref{eq:theorem33}) we have
\begin{eqnarray}
a_{j_0}'+a_{j_0+1}\equiv b_{j_0}'k_{1}^{j_0}+b_{j_0+1}k_{1}^{j_0+1}\quad(\mathrm{mod} \quad p)\nonumber\\
2^{j_0}a_{j_0}'+2^{j_0+1}a_{j_0+1}\equiv b_{j_0}'k_{2}^{j_0}+b_{j_0+1}k_{2}^{j_0+1}\quad(\mathrm{mod} \quad p).\nonumber
\end{eqnarray}
Without loss of generality, we can assume $\epsilon_{j_0}'=1$ and $\epsilon_{j_0}''=0$ and then
\begin{eqnarray}
a_{j_0}+a_{j_0+1}\equiv b_{j_0+1}k_{1}^{j_0+1}\quad(\mathrm{mod} \quad p)\nonumber\\
2^{j_0}a_{j_0}+2^{j_0+1}a_{j_0+1}\equiv b_{j_0+1}k_{2}^{j_0+1}\quad(\mathrm{mod} \quad p), \nonumber
\end{eqnarray}
which implies  $a_{j_0}=b_{j_0}=0$ since $k_{2}\equiv 2k_{1}\quad(\mathrm{mod} \quad p)$ and
$a_{j_{0}}\equiv b_{j_{0}}k_{1}^{j_{0}}\quad(\mathrm{mod}\quad p)$.

(2) We first assume that ${\mathcal Z}=\emptyset$ which implies  $a_j=b_j=0$ provided $\epsilon_{j}'\neq\epsilon_{j}''$.
Take
 \begin{equation}
\epsilon_j=
\left\{
\begin{array}{c}
\epsilon_j',\quad \text{ if }\epsilon_j'=\epsilon_j''\\
0, \quad \text{ if }\epsilon_j'\neq \epsilon_j''
\end{array}\right. .
\nonumber
\end{equation}
Noting that $a_j=b_j=0$ provided $\epsilon_j'\neq \epsilon_j''$, we have
 $\mathcal{P}^{\va, \epsilon}_{d, p}=\mathcal{P}^{\va, \epsilon'}_{d, p}$ and $\mathcal{P}^{\vb, \epsilon}_{d, p}=\mathcal{P}^{\vb, \epsilon''}_{d, p}$.
 The (2) of Theorem 2.4 implies that $\mathcal{P}^{\va, \epsilon'}_{d, p}\cap\mathcal{P}^{\vb, \epsilon''}_{d, p}=\left\{(0, \ldots, 0)\right\}$. We next consider the case where ${\mathcal Z}\neq \emptyset$.
Suppose that $\mathcal{P}^{\va, \epsilon'}_{d, p}$ and $\mathcal{P}^{\vb, \epsilon''}_{d, p}$ have a common nonzero point.
Then, there exist $j, k\in\mathbb{Z}_{p}^{*}$ so that
\begin{equation}\label{eq:fang}
\begin{aligned}
a_{1}j&\equiv b_{1}k\quad(\mathrm{mod} \quad p)\\
a_{1}'j+a_{2}j^{2}&\equiv b_{1}'k+b_{2}k^{2}\quad(\mathrm{mod} \quad p)\\
&\vdots\\
\sum_{h=1}^{d-1}a_{h}'j^{h}+a_{d}j^{d}&\equiv \sum_{h=1}^{d-1}b_{h}'k^{h}+b_{d}k^{d}\quad(\mathrm{mod} \quad p).
\end{aligned}
\end{equation}
Note that $a_h=b_h=0$ when $h\leq r-1$ and  $a_h'=b_h'=0, h\leq r-1$ .
The (\ref{eq:fang}) implies that
\begin{equation}\label{eq:theorem36}
a_{h}j^{h}\equiv b_{h}k^{h}\quad(\mathrm{mod} \quad p), \quad h=r, \ldots, \ell_0,
\end{equation}
\begin{equation}\label{eq:theorem37}
a_{\ell}'j^{\ell_0}+a_{\ell_0+1}j^{\ell_0+1}\equiv b_{\ell_0}'k^{\ell_0}+b_{\ell_0+1}k^{\ell_0+1}\quad(\mathrm{mod} \quad p).
\end{equation}
Without loss of generality, we can assume $\epsilon_{\ell_0}'=1$ and $\epsilon_{\ell_0}''=0$. By (\ref{eq:theorem37}), we have
\begin{equation}
a_{\ell_0}j^{\ell_0}+a_{\ell_0+1}j^{\ell_0+1}\equiv b_{\ell_0+1}k^{\ell_0+1}\quad(\mathrm{mod} \quad p).\nonumber
\end{equation}
Taking $h=r$ in  (\ref{eq:theorem36}),
we have $a_{r}\equiv b_{r}(kj^{-1})^{r}\quad(\mathrm{mod} \quad p)$ where $j^{-1}\in \Z_p$ satisfies  $j^{-1} j\equiv 1\quad (\mathrm{mod}\quad p)$.
Since  $a_{r}^{2}+b_{r}^{2}\neq 0$, we have $a_{r}\neq 0$. Set $x_0=kj^{-1}$. Then $x_0$ satisfies
 \begin{equation*}
 \begin{aligned}
 a_{r}& \equiv b_{r}x_0^{r}\quad(\mathrm{mod} \quad p)\\
 a_{\ell_0}j^{-1}+a_{\ell_0+1} & \equiv b_{\ell_0+1}x_0^{\ell_0+1}\quad(\mathrm{mod} \quad p).
 \end{aligned}
 \end{equation*}
 Each nonzero point in $\mathcal{P}^{\va, \epsilon'}_{d, p}\cap\mathcal{P}^{\vb, \epsilon''}_{d, p}$
 corresponds to a solution to
  \begin{equation}\label{eq:fangcheng}
 \begin{aligned}
 a_{r}& \equiv b_{r}x^{r}\quad(\mathrm{mod} \quad p)\\
 a_{\ell_0}j^{-1}+a_{\ell_0+1} & \equiv b_{\ell_0+1}x^{\ell_0+1}\quad(\mathrm{mod} \quad p).
 \end{aligned}
 \end{equation}

Note that $a_{r}\equiv b_{r}x^{r}\quad(\mathrm{mod} \quad p)$ has at most $r$ solutions. Hence,
 \[
 \abs{\mathcal{P}^{\va, \epsilon'}_{d, p}\cap\mathcal{P}^{\vb, \epsilon''}_{d, p}}\,\,\leq\,\, r+1.
 \]

(3)
We prove it by contradiction. Assume that $\mathcal{P}_{d, p}^{\va,\epsilon'}\cap\mathcal{P}_{d, q}^{\va,\epsilon''}\neq\left\{(0, \ldots, 0)\right\}$, and then there exist
$j_{0}, k_{0}\in\mathbb{Z}_{p}^{*}$ so that $\mathbf{x}_{j_{0}}^{\va,\epsilon'}=\mathbf{x}_{k_{0}}^{\vb,\epsilon''}$. Particularly, we have
\begin{equation}\label{eq:15}
a_{h}j_{0}^{h}\equiv b_{h}k_{0}^{h}\quad(\mathrm{mod} \quad p), h=1, \ldots, \ell_{0},
\end{equation}
\begin{equation}\label{eq:16}
a_{\ell_{0}}'j_{0}^{\ell_{0}}+a_{\ell_{0}+1}j_{0}^{\ell_{0}+1}\equiv b_{\ell_{0}}'k_{0}^{\ell_{0}}+b_{\ell_{0}+1}k_{0}^{\ell_{0}+1} \quad(\mathrm{mod} \quad p).
\end{equation}
Without loss of generality, we can assume $\epsilon_{\ell_0}'=1$ and $\epsilon_{\ell_0}''=0$. By (\ref{eq:16}), we have
\begin{equation}\label{eq:17}
a_{\ell_{0}}j_{0}^{\ell_{0}}+a_{\ell_{0}+1}j_{0}^{\ell_{0}+1}\equiv b_{\ell_{0}+1}k_{0}^{\ell_{0}+1} \quad(\mathrm{mod} \quad p).
\end{equation}
By (\ref{eq:15}) with $h=1$, we have
\begin{eqnarray}
a_{\ell_{0}+1}j_{0}^{\ell_{0}+1}-b_{\ell_{0}+1}k_{0}^{\ell_{0}+1} &\equiv&
a_{\ell_{0}+1}(b_{1}k_{0}a_{1}^{-1})^{\ell_{0}+1}-b_{\ell_{0}+1}k_{0}^{\ell_{0}+1}\nonumber\\
&\equiv& k_{0}^{\ell_{0}+1}(a_{\ell_{0}+1}b_{1}^{\ell_{0}+1}a_{1}^{-\ell_{0}-1}-b_{\ell_{0}+1})\nonumber\\
&\equiv& 0 \quad(\mathrm{mod} \quad p),\nonumber
\end{eqnarray}
according to $a_{\ell_0+1}b_{1}^{\ell_0+1}\equiv a_{1}^{\ell_0+1}b_{\ell_0+1}\quad(\mathrm{mod} \quad p)$. By (\ref{eq:17}), we have
$a_{\ell_{0}}j_{0}^{\ell_{0}}\equiv 0\quad(\mathrm{mod} \quad p)$, which implies that $a_{\ell_{0}}\equiv 0\quad(\mathrm{mod} \quad p)$
or $j_{0}\equiv 0\quad(\mathrm{mod} \quad p)$. This is impossible by the assumption.

\end{proof}

In the following, we choose the appropriate vectors $\va, \vb$ so that  $|\mathcal{L}_{p, q}|=q+p-1$.
We now state the inequalities for exponential sums over ${\mathcal L}_{p,q}$, which is the main result of this subsection.

\begin{theorem}\label{th:weil2}
Suppose $p$ and $q$ are odd prime numbers and set $m=p+q$. Recall that
\begin{equation}
\mathcal{L}_{p, q}=\left\{
\begin{array}{c}
\mathcal{P}_{d,p}\cup\mathcal{P}_{d,q}, \quad p\neq q\\
\mathcal{P}^{\va, \epsilon'}_{d, p}\cup
\mathcal{P}^{\vb, \epsilon''}_{d, p}, \quad p=q.
\end{array}\right.\nonumber
\end{equation}
We assume that $\abs{\mathcal{L}_{p,q}}=p+q-1$.
Then, for any $\mathbf{k}\in[-p+1, p-1]^{d}\cap[-q+1, q-1]^{d}\cap\mathbb{Z}^{d}$ and $\mathbf{k}\neq 0$, we have
\begin{equation}
\Big|\sum_{\mathbf{x}\in\mathcal{L}_{p, q}}\exp(2\pi\mathbf{i}\mathbf{k}\cdot\mathbf{x})\Big|\leq(d-1)\sqrt{2m}+1.
\nonumber
\end{equation}
\end{theorem}
\begin{proof}
We first consider the case where $p=q$. We have
\[
{\mathcal L}_{p,q}=\mathcal{P}^{\va, \epsilon'}_{d, p}\cup
\mathcal{P}^{\vb, \epsilon''}_{d, p}.
\]
Recall that
\[
\mathcal{P}^{\va, \epsilon'}_{d, p}\cap\mathcal{P}^{\vb, \epsilon''}_{d, p}=\left\{(0, \ldots, 0)\right\}.
\]
Then
\begin{eqnarray}
\Big|\sum_{\mathbf{x}\in\mathcal{L}_{p, q}}\exp(2\pi\mathbf{i}\mathbf{k}\cdot\mathbf{x})\Big|
&\leq & \Big|\sum_{\mathbf{x}\in\mathcal{P}_{d,p}^{\va, \epsilon'}}\exp(2\pi\mathbf{i}\mathbf{k}\cdot\mathbf{x})\Big|
+\Big|\sum_{\mathbf{x}\in\mathcal{P}_{d,p}^{\vb, \epsilon''}}\exp(2\pi\mathbf{i}\mathbf{k}\cdot\mathbf{x})\Big|+1\nonumber\\
&\leq &(d-1)\sqrt{p}+(d-1)\sqrt{p}+1\nonumber\\
&=&(d-1)\sqrt{2m}+1\nonumber.
\end{eqnarray}
Here, in the last inequality, we use  Theorem \ref{th:weil1}.
We next consider the case where $p\neq q$.
When $p\neq q$, $\mathcal{L}_{p, q} = \mathcal{P}_{d,p}\cup\mathcal{P}_{d,q}$. Then we have
\begin{eqnarray}
\Big|\sum_{\mathbf{x}\in\mathcal{L}_{p, q}}\exp(2\pi\mathbf{i}\mathbf{k}\cdot\mathbf{x})\Big|&\leq & \Big|\sum_{\mathbf{x}\in\mathcal{P}_{d,p}}\exp(2\pi\mathbf{i}\mathbf{k}\cdot\mathbf{x})\Big|
+\Big|\sum_{\mathbf{x}\in\mathcal{P}_{d,q}}\exp(2\pi\mathbf{i}\mathbf{k}\cdot\mathbf{x})\Big|+1\nonumber\\
&\leq & (d-1)\sqrt{p}+(d-1)\sqrt{q}+1\nonumber\\
&\leq &(d-1)\sqrt{2m}+1\nonumber.
\end{eqnarray}
\end{proof}
\subsection{The exponential sums over ${\mathcal Q}_{p^2,d}^{\va,\epsilon}$ and ${\mathcal R}_{p^2,d}^{\va,\epsilon}$}Suppose that $\va\in {\mathbb Z}_p^d$ and $\epsilon\in \left\{0,1\right\}^{d-1}$. We set
\begin{eqnarray*}
{\mathcal Q}_{p^2,d}^{\va,\epsilon}:=\left\{\mathbf{z}_{j}^{\va,\epsilon} : j=0,\ldots,p^2-1\right\},
\end{eqnarray*}
\begin{eqnarray*}
\mathbf{z}_{j}^{\va,\epsilon}=\Big(\left\{\frac{a_1j}{p^2}\right\}, \left\{\frac{a_1'j+a_2j^{2}}{p^2}\right\}, \ldots, \left\{\frac{\sum_{h=1}^{d-1}a_h'j^h+a_dj^{d}}{p^2}\right\}\Big)\in [0,1)^d;
\end{eqnarray*}
and
\begin{eqnarray*}
{\mathcal R}_{p^2,d}^{\va,\epsilon}:=\left\{\vz_{j,k}^{\va,\epsilon}: j,k=0,\ldots,p-1\right\},
\end{eqnarray*}
\begin{eqnarray*}
\vz_{j,k}^{\va,\epsilon}=\Big(\left\{\frac{a_1k}{p} \right\},\left\{\frac{(a_1'+a_2j)k}{p} \right\},\ldots,\left\{\frac{(\sum_{h=1}^{d-1}a_h'j^{h-1}+a_{d}j^{d-1})k}{p} \right\}\Big)\in [0,1)^d.
\end{eqnarray*}
The ${\mathcal Q}_{p^2,d}^{\va,\epsilon}$ and ${\mathcal R}_{p^2,d}^{\va,\epsilon}$ can be considered as
the generalization of the $p$-sets given in $(\ref{eq:p2set})$.
Based on the Lemma 5 and Lemma 6 in \cite{Dick2}, we can obtain the following inequalities for exponential sums
over ${\mathcal Q}_{p^2,d}^{\va,\epsilon}$ and ${\mathcal R}_{p^2,d}^{\va,\epsilon}$.
\begin{theorem}
Suppose that $\va\in {\mathbb Z}_p^d$ and $\epsilon\in \left\{0,1\right\}^{d-1}$. Then, for any $\mathbf{k}=(k_1,\ldots,k_d)\in[-p+1, p-1]^{d}\cap\mathbb{Z}^{d}$ and $\mathbf{k}\neq 0$, we have
\begin{equation}
\Big|\sum_{\mathbf{x}\in{\mathcal Q}_{p^2,d}^{\va,\epsilon}}\exp(2\pi\mathbf{i}\mathbf{k}\cdot\mathbf{x})\Big|\leq(d-1)p.
\nonumber
\end{equation}
\end{theorem}
\begin{proof}
Set
\begin{equation*}
g(j)\,\,:=\,\,\sum_{\ell=1}^{d}c_{\ell}j^{\ell},
\end{equation*}
where $c_{\ell}=k_{\ell}a_{\ell}+k_{\ell+1}a_{\ell}'+\cdots+k_{d}a_{\ell}'$.
We set  $j_0:=\max\left\{\ell:k_{\ell}\neq 0\right\}$. Then $c_{j_0}=k_{j_0}a_{j_0}$ and we have $p\nmid c_{j_0}$.
According to Lemma 5 in \cite{Dick2}, we have
\begin{equation*}
\Big|\sum_{\mathbf{x}\in{\mathcal Q}_{p^2,d}^{\va,\epsilon}}\exp(2\pi\mathbf{i}\mathbf{k}\cdot\mathbf{x})\Big|
 =\Big|\sum_{j=0}^{p^{2}-1}\exp\Big(2\pi\mathbf{i}\frac{g(j)}{p^{2}}\Big)\Big|\leq (d-1)p.
\end{equation*}
\end{proof}
\begin{theorem}
Suppose that $\va\in[1, p-1]^{d}\cap\mathbb{Z}^{d}$. Then, for any $\mathbf{k}\in[-p+1, p-1]^{d}\cap\mathbb{Z}^{d}$ and $\mathbf{k}\neq 0$, we have
\begin{equation}
\Big|\sum_{\mathbf{x}\in{{\mathcal R}_{p^2,d}^{\va,\epsilon}}}\exp(2\pi\mathbf{i}\mathbf{k}\cdot\mathbf{x})\Big|\leq(d-1)p.
\nonumber
\end{equation}
\end{theorem}
\begin{proof}
Set
\begin{equation*}
g(j):=\sum_{\ell=0}^{d-1}c_{\ell}j^{\ell}
\end{equation*}
and $c_{\ell}:=k_{\ell+1}a_{\ell+1}+k_{\ell+2}a_{\ell+1}'+\cdots+k_{d}a_{\ell+1}'$.
We set  $j_0:=\max\left\{\ell:k_{\ell}\neq 0\right\}$. Then $c_{j_0-1}=k_{j_0}a_{j_0}$ and we have $p\nmid c_{j_0-1}$.
Using  Lemma 6 in \cite{Dick2}, we have
\begin{equation}
\Big|\sum_{\mathbf{x}\in{{\mathcal R}_{p^2,d}^{\va,\epsilon}}}\exp(2\pi\mathbf{i}\mathbf{k}\cdot\mathbf{x})\Big| =\Big|\sum_{j=0}^{p-1}\sum_{k=0}^{p-1}\exp\Big(2\pi\mathbf{i}k\frac{g(j)}{p}\Big)\Big|\leq (d-1)p.\nonumber
\end{equation}

\end{proof}
\section{The applications of $\mathcal{P}_{d,p}^{\va,\epsilon}$ and
${\mathcal L}_{p,q}$}
Based on the exponential sum  formula in Section 2, the new point sets are useful
in numerical integration \cite{Nie, Dick3}, in UQ  \cite{XZ} and in the recovery of sparse trigonometric polynomials \cite{Xu}. We just state the results for the recovery of sparse trigonometric polynomials in detail.

We start with some notations which go back to \cite{Xu}. Set
\[
\Pi_{s}^{d}\,\,:=\,\,\left\{f: f(\mathbf{x})=\sum_{\mathbf{k}\in[-s,s]^{d}\bigcap\mathbb{Z}^{d}}c_{\mathbf{k}}e^{2\pi \mathbf{i}\mathbf{k}\cdot\mathbf{x}},\quad c_{\mathbf{k}}\in\mathbb{C}, \quad\mathbf{x}\in[0,1]^{d}\right\}.
\]
Note that $\Pi_{s}^{d}$ is a linear space with the dimension $D:=(2s+1)^{d}$. For
\[
f(\mathbf{x})=\sum_{\mathbf{k}\in[-s,s]^{d}\bigcap\mathbb{Z}^{d}}c_{\mathbf{k}}e^{2\pi \mathbf{i}\mathbf{k}\cdot\mathbf{x}}\in \Pi_{s}^{d},
\]
we set $\mathbf{T}:=\left\{\mathbf{k}:c_{\mathbf{k}}\neq 0\right\}$ which is the support of the sequence of coefficients $c_{\mathbf{k}}$, and set
\[
\Pi_{s}^{d}(M):=\bigcup_{\mathbf{T}\subset[-s,s]^{d}\bigcap\mathbb{Z}^{d},|\mathbf{T}|\leq M}\Pi_{\mathbf{T}},
\]
where $\Pi_{\mathbf{T}}$ denotes the space of all trigonometric polynomials whose coefficients are supported on $\mathbf{T}$. When $M\ll D$, we call the trigonometric polynomials in
$\Pi_s^{d}(M)$ as $M$-sparse trigonometric polynomials.

The recovery of sparse trigonometric polynomials is an active topic recently.
The main aim of this research topic is to design a sampling set $X=\left\{\vz_j\right\}_{j=1}^N$ so that
one can recover $f\in \Pi_s^d(M)$ from $f(\vz_j), \vz_j\in X$ \cite{Xu, R07, Rauhut}. We state the problem as follows.
Assume  the sampling set is $X=\left\{\mathbf{x}_{j}\in[0, 1)^{d}, j=1, \ldots, N\right\}$. Then our aim is to solve the following programming:
\begin{equation}\label{eq:re1}
\text{ find } f \in\Pi_{s}^{d}(M) \qquad \text{ subject to } \quad
f(\mathbf{x}_{j})=y_j,\quad j=1, \ldots, N.
\end{equation}
Denote by $\mathbf{F}_{X}$ the $N\times D$ sampling matrix with entries
\[
(\mathbf{F}_{X})_{j,\mathbf{k}}=\exp(2\pi \mathbf{i}\mathbf{k}\cdot \mathbf{x}_{j}),\quad j=1,\ldots, N,\quad \mathbf{k}\in[-s,s]^{d}\cap\mathbb{Z}^{d}.
\]
Let $\mathbf{a}_{\mathbf{k}}=(\exp(2\pi \mathbf{i}\mathbf{k}\cdot \mathbf{x}_{j}))_{j=1}^N$ denote a column of
$\mathbf{F}_{X}$ with $\mathbf{k}\in[-s,s]^{d}\bigcap\mathbb{Z}^{d}$. A simple observation is that $\|\mathbf{a}_{\mathbf{k}}\|_{2}=\sqrt{N}$. Set
\begin{equation}
\mu:=\mu_X:=\frac{1}{N}\max_{\mathbf{m}, \mathbf{k}\in[-s,s]^{d}\cap\mathbb{Z}^{d},
\mathbf{m}\neq\mathbf{k}}|\langle\mathbf{a}_{\mathbf{m}}, \mathbf{a}_{\mathbf{k}}\rangle|,
\nonumber
\end{equation}
which is called the mutual incoherence of the matrix $\mathbf{F}_{X}/\sqrt{N}$. Theorem 2.5 in \cite{Rauhut} shows that if $\mu<1/(2M-1)$ then
the Orthogonal Matching Pursuit Algorithm (OMP) and the Basis Pursuit Algorithm (BP) can recover any $M$-sparse trigonometric polynomials in $\Pi_{s}^{d}(M)$.
Therefore, our aim is to  choose the sampling set $X$ so that $\mu$ is small and hence OMP and BP can recover $M$-sparse trigonometric polynomials. Based on Theorem \ref{th:weil1} and Theorem \ref{th:weil2} respectively, the following results give upper bounds of $\mu$ with taking $X=\mathcal{P}_{d,p}^{\va, \epsilon}$,
and $X={\mathcal L}_{p,q}$, respectively.

\begin{lemma}\label{th:coh}
\begin{enumerate}
\item Suppose that $X=\mathcal{P}_{d,p}^{\va, \epsilon}$ where $\va\in[1, p-1]^{d}\cap\mathbb{Z}^{d}$ and $p\geq 2s+1$ is a prime number. Then
\begin{equation}
\mu_X\,\,\leq\,\, (d-1)/\sqrt{p}.
\nonumber
\end{equation}
\item
Suppose that $p,q\geq 2s+1$ are prime numbers and
$\va, \vb\in[1, p-1]^{d}\cap\mathbb{Z}^{d}$.
Recall that
\[
\mathcal{L}_{p, q}=\left\{
\begin{array}{c}
\mathcal{P}_{d,p}\cup\mathcal{P}_{d,q}, \quad p\neq q\\
\mathcal{P}^{\va, \epsilon'}_{d, p}\cup
\mathcal{P}^{\vb, \epsilon''}_{d, p}, \quad p=q.
\end{array}\right.
\]
Set $X=\mathcal{L}_{p, q}$ and $m=p+q$. Then
\begin{equation}
\mu_X\leq\frac{(d-1)\sqrt{2m}+1}{m-1}.
\nonumber
\end{equation}
\end{enumerate}
\end{lemma}
As said before, if  $\mu<1/(2M-1)$ then OMP (and also BP) can recover every $M$-sparse trigonometric polynomials. Then we have the following corollary:
\begin{theorem}
\begin{enumerate}
\item Suppose that $p>\max\left\{2s+1, (d-1)^{2}(2M-1)^{2}+1\right\}$ is a prime number
and $\va\in[1, p-1]^{d}\cap\mathbb{Z}^{d}$ . Then OMP (and also BP) recovers every $M$-sparse trigonometric polynomial $f\in \Pi_{s}^{d}(M)$ exactly from the deterministic sampling $\mathcal{P}_{d, p}^{\va, \epsilon}$.
\item Under the condition in (2) of Lemma \ref{th:coh}.
Suppose that
\[
m=p+q > \left(\left(\frac{1}{\sqrt{2}}+\frac{1}{2}\right) (2M-1)(d-1)+\sqrt{M}\right)^2.
\]
Then OMP (and also BP) recovers every $M$-sparse trigonometric polynomial $f\in \Pi_{s}^{d}(M)$ exactly from the deterministic sampling set $\mathcal{L}_{p, q}$ .
\end{enumerate}
\end{theorem}
\begin{proof}
We first consider (1).
Note that $p\geq (d-1)^{2}(2M-1)^{2}+1$ implies that  $(d-1)/\sqrt{p}<1/(2M-1)$.
According to (1) in Lemma \ref{th:coh}, if $(d-1)/\sqrt{p}<1/(2M-1)$ then $\mu<1/(2M-1)$ and hence the conclusion follows.  Similarly, we can prove (2).
\end{proof}

\end{document}